\documentclass[12pt]{amsart}
\usepackage{amssymb}

\newtheorem{theorem}{Theorem}[section]
\newtheorem{definition}[theorem]{Definition}

\newtheorem{proposition}[theorem]{Proposition}

\newtheorem{lemma}[theorem]{Lemma}

\begin{document}

\title[Cyclotomic expansion]{Cyclotomic expansion of exceptional spectral measures}

\author{Teodor Banica}
\address{Department of Mathematics, Toulouse 3 University, 118 route de Narbonne, 31062 Toulouse, France. {\tt banica@math.ups-tlse.fr}}

\thanks{Work supported by the ANR grant 07-BLAN-0229}

\begin{abstract}
We find explicit formulae for the circular spectral measures of $E_7,E_8$. This leads to a number of general observations regarding the ADE circular measures: these are linear combinations of measures supported by the roots of unity, with real density given by certain degree 3 polynomials.
\end{abstract}

\maketitle

\section*{Introduction}

One of the main problems in quantum groups or subfactors is the computation of a certain real probability measure $\mu$. In the case of a compact group $G\subset U_n$, this is the spectral measure of the character of the fundamental representation with respect to the Haar integration, with moments given by:
$$c_k=\int_G {\rm Tr}(g)^k\,dg$$

In the case of the dual of a discrete group $\Gamma=<g_1,\ldots,g_n>$, this is the Kesten measure of the group, whose moments are given by:
$$c_k=\#\{i_1,\ldots,i_k\mid g_{i_1}\ldots g_{i_k}=1\}$$

Observe that, by standard Fourier analysis, the above two formulae coincide indeed in the case of a pair $(G,\Gamma)$ of dual abelian groups.

In the general case of a compact quantum group $G={\rm Spec}(A)$, the numbers $c_k$ and the measure $\mu$ can be defined by similar formulae.

An even more general situation is that of a subfactor $N\subset M$. The symmetries of $N\subset M$ are encoded by a group-like object $G$, and the Cayley graph of $\widehat{G}$ is a certain rooted bipartite graph $(X,*)$, called principal graph. The numbers $c_k$ can be recovered by counting the lenght $2k$ loops based at the root:
$$c_k={\rm loop}_X(2k)$$

As for the measure $\mu$ itself, this is nothing but the spectral measure of the adjacency matrix of $X$, with respect to linear form $a\to a_{**}$.

In general, the computation of $\mu$ is a quite delicate problem, and complete results are available only for a few groups, group duals, quantum groups or subfactors. For instance in the group case  the computation of $\mu$ requires a good knowledge of the representation theory of $G$. As for the group dual case, the difficulty in computing the Kesten measure is also well-known.

A number of new ideas have appeared in the recent years in connection with the above problem, formulated for quantum groups or subfactors. In all cases, the change of variables $\Phi(q)=(q+q^{-1})^2$ in the complex plane seems to play a key role. The story of this change of variables is as follows:

\begin{enumerate}
\item In \cite{gle} Graham and Lehrer classified the representations of the affine Temperley-Lieb algebra. In \cite{jo2} Jones used a version of their result, in order to describe the irreducible Temperley-Lieb planar modules. As a corollary, in the $[M:N]>4$ case the theta series (obtained from the Stieltjes transform of $\mu$ via the above change of variables) has the remarkable property that all its coefficients are positive numbers.

\item In the $[M:N]\leq 4$ case, corresponding to the ADE principal graphs, a systematic study of the theta series was done by Reznikoff in \cite{rez}. A direct probabilistic approach to the computation and interpretation of the theta series was taken in our joint paper with Bisch \cite{bdb}, with the main result that the measure $\varepsilon=\Phi_*\mu$ is given by very simple formulae.

\item The change of variables $\Phi$ appears as well in the context of Wang's quantum permutation groups \cite{wan}. In our joint paper with Bichon \cite{bjb}, we describe the subgroups of $S_4^+$, and we find an ADE classification result for them. The correspondence, which is much less explicit than McKay's correspondence \cite{mck}, is in fact obtained by using circular measures $\varepsilon$.

\item Now back to subfactors, some exciting new results in the $[M:N]>4$ case come from the general observations of Coste and Gannon in \cite{cga}, via the obstruction formulated by Etingof, Nikshych and Ostrik in \cite{eno}. The applications, due to Asaeda \cite{asa} and Asaeda and Yasuda \cite{aya}, are based on a number of concrete computations, where the change of variables $\Phi$ is present as well, at least in some hidden form.

\item As for the recent trends in compact quantum groups, once again we seem to get into $\Phi$. In a joint paper with Collins \cite{bco} we worked out the Weingarten formula for $O_n^+$, in terms of Di Francesco's meander determinants \cite{dif}. The computation of the law of $u_{11}$, left open at that time, can be in fact solved, by computing the associated circular measure $\varepsilon$. This will be discussed in a forthcoming paper with Collins and Zinn-Justin \cite{bcz}.
\end{enumerate}

Summarizing, the change of variables $\Phi(q)=(q+q^{-1})^2$ and the circular spectral measure $\varepsilon=\Phi_*\mu$ seem to play an increasing role in a number of recent considerations. This confirms the hopes and expectations formulated in \cite{bdb}.

The purpose of this paper is to go back to the investigations in \cite{bdb}, with the will of systematically improving the material there. We have two results here: (1) an explicit formula for the measures of $E_7,E_8$, and (2) a number of general observations regarding the class of measures associated to the ADE graphs.

These results, complementary to those in \cite{bdb}, are expected to be of help in connection with the above-mentioned directions of research.

The paper is organized as follows: 1-3 are preliminary sections, in 4-6 we discuss the notion of cyclotomic measure, and in 7-8 we study the measures associated to the ADE graphs. The final section, 9, contains a few concluding remarks.

\section{Principal graphs}

A subfactor is an inclusion of ${\rm II}_1$ factors $N\subset M$. The index of the subfactor is the number $\lambda=\dim_N(M)$. The dimension is taken in the Murray-von Neumann sense, and we have $\lambda\in[1,\infty]$. In what follows we assume $\lambda<\infty$.

The basic construction associates to $N\subset M$ a certain subfactor $M\subset M_2$, by a kind of general mirroring procedure. By iterating the basic construction we obtain the Jones tower $N\subset M\subset M_2\subset M_3\subset\ldots$. See \cite{jo1}.

The higher relative commutants of the subfactor are the algebras $P_k=N'\cap M_k$. Since the index is finite, we have $\dim(P_k)<\infty$. The system of algebras $P=(P_k)$ carries a rich combinatorial structure, called planar algebra \cite{jo2}. In the amenable case, $P$ is a complete invariant for the subfactor \cite{pop}.

The principal graph $X$ of the subfactor is obtained by taking the Bratelli diagram of the system of inclusions $P_0\subset P_1\subset P_2\subset\ldots$, then by deleting the reflections coming from basic constructions. See \cite{jo1}.

The main properties of $X$ can be summarized as follows.

\begin{proposition}
The principal graph $X$ has the following properties:
\begin{enumerate}
\item In the amenable case, the index of $N\subset M$ is given by $\lambda=||X||^2$.

\item The higher relative commutant $P_k=N'\cap M_k$ is isomorphic to the abstract vector space spanned by the $2k$-loops on $X$ based at the root.
\end{enumerate} 
\end{proposition}

It follows from this result that the principal graph of a subfactor having index $\lambda\leq 4$ must be one of the Coxeter-Dyknin graphs of type ADE. See \cite{ghj}.

The usual Coxeter-Dynkin graphs are as follows:
$$A_n=\bullet-\circ-\circ\cdots\circ-\circ-\circ\hskip20mm A_{\infty}=\bullet-\circ-\circ-\circ\cdots\hskip7mm$$
\vskip-7mm
$$D_n=\bullet-\circ-\circ\dots\circ-
\begin{matrix}\ \circ\cr\ |\cr\ \circ \cr\ \cr\  \end{matrix}-\circ\hskip70mm$$
\vskip-7mm
$$\ \ \ \ \ \ \ \tilde{A}_{2n}=
\begin{matrix}
\circ&\!\!\!\!-\circ-\circ\cdots\circ-\circ-&\!\!\!\!\circ\cr
|&&\!\!\!\!|\cr
\bullet&\!\!\!\!-\circ-\circ-\circ-\circ-&\!\!\!\!\circ\cr\cr\cr\end{matrix}\hskip15mm A_{-\infty,\infty}=
\begin{matrix}
\circ&\!\!\!\!-\circ-\circ-\circ\cdots\cr
|&\cr
\bullet&\!\!\!\!-\circ-\circ-\circ\cdots\cr\cr\cr\end{matrix}
\hskip15mm$$
\vskip-9mm
$$\;\tilde{D}_n=\bullet-
\begin{matrix}\circ\cr|\cr\circ\cr\ \cr\ \end{matrix}-\circ\dots\circ-
\begin{matrix}\ \circ\cr\ |\cr\ \circ \cr\ \cr\  \end{matrix}-\circ \hskip20mm D_\infty=\bullet-
\begin{matrix}\circ\cr|\cr\circ\cr\ \cr\ \end{matrix}-\circ-\circ\cdots\hskip7mm$$
\vskip-7mm

Here the graphs $A_n$ with $n\geq 2$ and $D_n$ with $n\geq 3$ have $n$ vertices each, $\tilde{A}_{2n}$ with $n\geq 1$ has $2n$ vertices, and $\tilde{D}_n$ with $n\geq 4$ has $n+1$ vertices.

The first graphs in each series are as follows:
$$A_2=\bullet-\circ\hskip10mm 
D_3=\begin{matrix}\ \circ\cr\ |\cr\ \bullet \cr\ \cr\  \end{matrix}-\circ \hskip10mm
\tilde{A}_2=\begin{matrix}
\circ\cr
||\cr
\bullet\cr&\cr&\cr\end{matrix}\hskip10mm 
\tilde{D}_4=\bullet-\!\!\!\!\!\begin{matrix}
\circ\hskip5mm \circ\cr
\backslash\ \,\slash\cr
\circ\cr&\cr&\cr\end{matrix}\!\!\!\!\!\!\!\!\!\!-\circ$$
\vskip-7mm

The exceptional Coxeter-Dynkin graphs are as follows:
$$E_6=\bullet-\circ-
\begin{matrix}\circ\cr|\cr\circ\cr\ \cr\ \end{matrix}-
\circ-\circ\hskip71mm$$
\vskip-13mm
$$E_7=\bullet-\circ-\circ-
\begin{matrix}\circ\cr|\cr\circ\cr\ \cr\ \end{matrix}-
\circ-\circ\hskip18mm$$
\vskip-15mm
$$\hskip30mm E_8=\bullet-\circ-\circ-\circ-
\begin{matrix}\circ\cr|\cr\circ\cr\ \cr\ \end{matrix}-
\circ-\circ$$
\vskip-19mm

$$\tilde{E}_6=\bullet-\circ-\begin{matrix}
\circ\cr|
\cr\circ\cr|&\cr\circ&\!\!\!\!-\ \circ\cr\ \cr\   \cr\ \cr\ \end{matrix}-\circ\hskip71mm$$
\vskip-22mm
$$\tilde{E}_7=\bullet-\circ-\circ-
\begin{matrix}\circ\cr|\cr\circ\cr\ \cr\ \end{matrix}-
\circ-\circ-\circ\hskip18mm$$
\vskip-15mm
$$\hskip30mm \tilde{E}_8=\bullet-\circ-\circ-\circ-\circ-
\begin{matrix}\circ\cr|\cr\circ\cr\ \cr\ \end{matrix}-
\circ-\circ$$
\vskip-5mm

The subfactors of index $\leq 4$ were intensively studied in the 80's and early 90's. About 10 years after the appearance of Jones' founding paper \cite{jo1}, a complete classification result was found, with contributions by many authors. 

A simplified form of this classification result is as follows.

\begin{theorem}
The principal graphs of subfactors of index $\leq 4$ are:
\begin{enumerate}
\item Index $<4$ graphs: $A_n$, $D_{even}$, $E_6$, $E_8$. 
\item Index $4$ finite graphs: $\tilde{A}_{2n}$, $\tilde{D}_n$, $\tilde{E}_6$, $\tilde{E}_7$, $\tilde{E}_8$.
\item Index $4$ infinite graphs: $A_\infty$, $A_{-\infty,\infty}$, $D_\infty$.
\end{enumerate}
\end{theorem}

Observe that the graphs $D_{odd}$ and $E_7$ don't appear in the above list. This is one of the subtle points of subfactor theory. See \cite{eka}.

\section{Analytic functions}

The Poincar\'e and theta series of a subfactor were introduced by Jones, mainly in order to deal with the case where the index is $>4$. See \cite{jo2}, \cite{jre}, \cite{rez}.

For the purposes of this paper, it is convenient to associate such series to any rooted bipartite graph.

\begin{definition}
The Poincar\'e series of a rooted bipartite graph $X$ is
$$f(z)=\sum_{k=0}^\infty{\rm loop}_X(2k)z^k$$
where ${\rm loop}_X(2k)$ is the number of $2k$-loops based at the root.
\end{definition}

In case $X$ is the principal graph of a subfactor $N\subset M$, Proposition 1.1 shows that $f$ is the Poincar\'e series of the subfactor, in the usual sense:
$$f(z)=\sum_{k=0}^\infty\dim(N'\cap M_k)z^k$$

The Poincar\'e series should be thought of as being a basic representation theory invariant of the underlying group-like object. For instance for the subfactor associated to a compact Lie group $G\subset U_n$, the Poincar\'e series is:
$$f(z)=\int_G\frac{1}{1-{\rm Tr}(g)z}\,dg$$

The theta series can introduced as a version of the Poincar\'e series, via the change of variables $z^{-1/2}=q^{1/2}+q^{-1/2}$.

\begin{definition}
The theta series of a rooted bipartite graph $X$ is
$$\Theta(q)=q+\frac{1-q}{1+q}f\left(\frac{q}{(1+q)^2}\right)$$
where $f$ is the Poincar\'e series.
\end{definition}

The theta series can be written as $\Theta(q)=\sum a_rq^r$, and it follows from the above formula that its coefficients are integers: $a_r\in\mathbb Z$.

In fact, we have the following explicit formula from \cite{jo2}, relating the coefficients of $\Theta(q)=\sum a_rq^r$ to those of $f(z)=\sum c_kz^k$:
$$a_r=\sum_{k=0}^r(-1)^{r-k}\frac{2r}{r+k}\begin{pmatrix}r+k\cr r-k\end{pmatrix}c_k$$

In case the theta series is that of a subfactor $N\subset M$ of index $\lambda>4$, it is known from \cite{jo2} that the numbers $a_r$ are in fact certain multiplicities associated to the planar algebra inclusion $TL_\lambda\subset P$. In particular, we have $a_r\in\mathbb N$.

In this paper we use a version of the theta series.

\begin{definition}
The $T$ series of a rooted bipartite graph $X$ is
$$T(q)=\frac{\Theta(q)-q}{1-q}$$
where $\Theta$ is the theta series.
\end{definition}

This normalization is there in order for the series to be additive with respect to a certain underlying measure $\varepsilon$. This will be explained later on.

In this paper we will be mainly concerned with the $T$ series of ADE graphs. These graphs have index $\leq 4$, and the planar algebra interpretation in \cite{jo2} doesn't work. The relevant combinatorics will be that of the roots of unity.

\begin{definition}
The series of the form
$$\xi(n_1,\ldots,n_s:m_1,\ldots,m_t)=\frac{(1-q^{n_1})\ldots(1-q^{n_s})}{(1-q^{m_1})\ldots(1-q^{m_t})}$$
with $n_i,m_i\in\mathbb N$ are called cyclotomic.
\end{definition}

It is convenient to allow as well $1+q^n$ factors, to be designated by $n^+$ symbols in the above writing. For instance we have $\xi(2^+:3)=\xi(4:2,3)$.

We use the notations $\xi'=\xi/(1-q)$ and $\xi''=\xi/(1-q^2)$. 

Finally, when one of the sets of indices is missing, we simply omit it. For instance, we have by definition $\xi(3:)=1-q^3$ and $\xi(:3)=1/(1-q^3)$.

The Poincar\'e series of the ADE graphs are given by quite complicated formulae. However, the corresponding $T$ series are all cyclotomic.

\begin{theorem}
The $T$ series of the ADE graphs are as follows:
\begin{enumerate}
\item For $A_{n-1}$ we have $T=\xi(n-1:n)$.

\item For $D_{n+1}$ we have $T=\xi(n-1^+:n^+)$.

\item For $\tilde{A}_{2n}$ we have $T=\xi'(n^+:n)$.

\item For $\tilde{D}_{n+2}$ we have $T=\xi''(n+1^+:n)$.

\item For $E_6$ we have $T=\xi(8:3,6^+)$.

\item For $E_7$ we have $T=\xi(12:4,9^+)$.

\item For $E_8$ we have $T=\xi(5^+,9^+:15^+)$.

\item For $\tilde{E}_6$ we have $T=\xi(6^+:3,4)$.

\item For $\tilde{E}_7$ we have $T=\xi(9^+:4,6)$.

\item For $\tilde{E}_8$ we have $T=\xi(15^+:6,10)$.
\end{enumerate}
\end{theorem}

These formulae are obtained in \cite{bdb}, by counting loops, then by making the change of variables $z^{-1/2}=q^{1/2}+q^{-1/2}$, and factorizing the resulting series.

An alternative proof for these formulae can be obtained by using planar algebra methods, see \cite{rez}. Some related computations appear as well in \cite{msm}.

\section{Circular measures}

We discuss here the measure-theoretic interpretation of the invariants discussed in the previous section. Once again, we start with an arbitrary rooted bipartite graph $X$. We introduce first a measure $\mu$, whose Stieltjes transform is $f$.

\begin{definition}
The real measure $\mu$ of a rooted bipartite graph $X$ is given by
$$f(z)=\int_0^\infty\frac{1}{1-xz}\,d\mu(x)$$
where $f$ is the Poincar\'e series.
\end{definition}

We have the following interpretation. Let $M$ be the adjacency matrix of the graph, and consider the square matrix $L=MM^t$. Also, let $<A>$ be the $(*,*)$-entry of a matrix $A$, where $*$ is the root. With these notations, we have:
\begin{eqnarray*}
f(z)
&=&\sum_{k=0}^\infty{\rm loop}_X(2k)z^k\\
&=&\sum_{k=0}^\infty\left<L^k\right>z^k\\
&=&\left<\frac{1}{1-Lz}\right>
\end{eqnarray*}

This shows that $\mu$ is in fact a spectral measure: $\mu={\rm law}(L)$.

In the subfactor case some interpretations are available as well. For instance in the situation coming from of a compact group $G\subset U_n$, discussed in the previous section, $\mu$ is the spectral measure of the character $g\to{\rm Tr}(g)$.

\begin{definition}
The circular measure $\varepsilon$ of a rooted bipartite graph $X$ is given by
$$d\varepsilon(q)=d\mu((q+q^{-1})^2)$$
where $\mu$ is the real measure.
\end{definition}

In other words, $\varepsilon$ is the pullback of $\mu$ via the map $\mathbb R\cup\mathbb T\to\mathbb R_+$ given by $q\to (q+q^{-1})^2$, where $\mathbb T$ is the unit circle. See \cite{bdb} for details.

As an example, assume that $\mu$ is a discrete measure, supported by $n$ positive numbers $x_1<\ldots<x_n$, with corresponding densities $p_1,\ldots,p_n$:
$$\mu=\sum_{i=1}^n p_i\delta_{x_i}$$

For each $i\in\{1,\ldots,n\}$ the equation $(q+q^{-1})^2=x_i$ has four solutions, that we can denote $q_i,q_i^{-1},-q_i,-q_i^{-1}$. With this notation, we have:
$$\varepsilon=\frac{1}{4}\sum_{i=1}^np_i(\delta_{q_i}+\delta_{q_i^{-1}}+\delta_{-q_i}+\delta_{-q_i^{-1}})$$

The basic properties of $\varepsilon$ can be summarized as follows.

\begin{proposition}
The circular measure has the following properties:
\begin{enumerate}
\item $\varepsilon$ has equal density at $q,q^{-1},-q,-q^{-1}$.

\item The odd moments of $\varepsilon$ are $0$.

\item The even moments of $\varepsilon$ are half-integers.

\item When $X$ has norm $\leq 2$, $\varepsilon$ is supported by the unit circle.

\item When $X$ is finite, $\varepsilon$ is discrete.

\item If $K$ is a solution of $L=(K+K^{-1})^2$, then $\varepsilon={\rm law}(K)$. 
\end{enumerate}
\end{proposition}

These results can be deduced from definitions, and are explained in \cite{bdb}. In addition, we have the following formula, which gives the even moments of $\varepsilon$.

\begin{proposition}
We have the Stieltjes transform type formula
$$2\int\frac{1}{1-qu^2}\,d\varepsilon(u)=1+T(q)(1-q)$$
relating the circular measure $\varepsilon$ to the $T$ series.
\end{proposition}

This result follows as well from definitions, the idea being to apply the change of variables $q\to (q+q^{-1})^2$ to the fact that $f$ is the Stieltjes transform of $\mu$:
\begin{eqnarray*}
2\int\frac{1}{1-qu^2}\,d\varepsilon(u)
&=&1+\frac{1-q}{1+q}f\left(\frac{q}{(1+q)^2}\right)\\
&=&1+\Theta(q)-q\\
&=&1+T(q)(1-q)
\end{eqnarray*}

We refer to \cite{bdb} for full details regarding this computation.

The above formula shows that, unlike the theta series, the $T$ series is additive in $\varepsilon$. That is, if $\alpha_i$ are scalars summing up to $1$, we have:
$$T_{\alpha_1\varepsilon_1+\ldots+\alpha_s\varepsilon_s}=\alpha_1T_{\varepsilon_1}+\ldots+\alpha_sT_{\varepsilon_s}$$

This justifies the choice of $T$ instead of $\Theta$, made in the previous section.

In the subfactor case, we have the following result, due to Jones \cite{jo2}.

\begin{theorem}
In case $X$ is the principal graph of a subfactor of index $>4$, the moments of $\varepsilon$ are positive numbers.
\end{theorem}

This follows indeed from the result in \cite{jo2} that the coefficients of $\Theta$ are positive numbers, via the identities in Definition 2.3 and Proposition 3.4.

As a first concrete example, we can compute the measure of $\tilde{A}_{2n}$.

\begin{proposition}
The circular measure of the graph $\tilde{A}_{2n}$ is the uniform measure on the $2n$-roots of unity.
\end{proposition}

Indeed, let us identify the vertices of $X=\tilde{A}_{2n}$ with the group $\{w^k\}$ formed by the $2n$-th roots of unity in the complex plane, where $w=e^{\pi i/n}$. The adjacency matrix of $X$ acts on functions $f\in C(X)$ in the following way:
$$Mf(w^s)=f(w^{s-1})+f(w^{s+1})$$

This shows that we have $M=K+K^{-1}$, where $K$ is given by:
$$Kf(w^s)=f(w^{s+1})$$

Thus we can use the last assertion in Proposition 3.3, and we get $\varepsilon={\rm law}(K)$, which is the uniform measure on the $2n$-roots. See \cite{bdb} for details.

\section{Cyclotomic measures}

We begin now a systematic study of the measures associated to the ADE graphs. The considerations in previous section suggest the following notion.

\begin{definition}
A cyclotomic measure is a probability measure $\varepsilon$ on the unit circle, having the following properties:
\begin{enumerate}
\item  $\varepsilon$ is supported by the $2n$-roots of unity, for some $n\in\mathbb N$.

\item $\varepsilon$ has equal density at $q,q^{-1},-q,-q^{-1}$.
\end{enumerate}
\end{definition}

It follows from Theorem 2.5 that the circular measures of the finite ADE graphs are supported by certain roots of unity, hence are cyclotomic. This fact will be discussed in detail later on, with explicit formulae for all graphs.

\begin{definition}
The $T$ series of a cyclotomic measure $\varepsilon$ is given by:
$$1+T(q)(1-q)=2\int\frac{1}{1-qu^2}\,d\varepsilon(u)$$
\end{definition}

Observe that this formula is nothing but the one in Proposition 3.4, written now in the other sense. In other words, if the cyclotomic measure $\varepsilon$ happens to be the circular measure of a rooted bipartite graph, then the $T$ series as defined above coincides with the $T$ series as defined in section 2.

We present now a number of generalities regarding cyclotomic measures. A first problem is to find the density in terms of the $T$ series.

\begin{definition}
We use the following notations:
\begin{enumerate}
\item $d_n$ is the uniform measure on the $2n$-roots of unity.

\item For $\lambda:\mathbb T\to\mathbb R$ we let $\lambda_n=\lambda\,d_n$.
\end{enumerate}
\end{definition}

In general $\lambda_n$ is not a probability measure. The example we are intereseted in is $\lambda(q)={\rm Re}(P(q^2))$, with $P\in\mathbb R[X]$. Since $\lambda$ has equal values at $q,q^{-1},-q,-q^{-1}$, $\lambda_n$ is a cyclotomic measure, provided that it is positive, and of mass 1.

\begin{lemma}
Let $P=1+a_1q+a_2q^2+\ldots+a_dq^d$ with $a_i\in\mathbb R$, and let $n>d$. In case $\varepsilon={\rm Re}(P(q^2))_n$ is a cyclotomic measure, its $T$ series is:
$$T=\frac{P(q)+q^nP(q^{-1})}{(1-q)(1-q^n)}$$
\end{lemma}

\begin{proof}
With the notation $a_0=1$, we have:
\begin{eqnarray*}
\int\frac{P(u^{-2})}{1-qu^2}\,d_nu
&=&\sum_{k=0}^\infty q^k\int P(u^{-2})u^{2k}d_nu\\
&=&\sum_{s=0}^da_s\sum_{k=0}^\infty q^k\int u^{2k-2s}d_nu\\
&=&\sum_{s=0}^da_s(q^{s}+q^{n+s}+q^{2n+s}+\ldots)\\
&=&\frac{P(q)}{1-q^n}
\end{eqnarray*}

Also, we have:
\begin{eqnarray*}
\int\frac{P(u^2)}{1-qu^2}\,d_nu
&=&\sum_{k=0}^\infty q^k\int P(u^2)u^{2k}d_nu\\
&=&\sum_{s=0}^da_s\sum_{k=0}^\infty q^k\int u^{2k+2s}d_nu\\
&=&1+\sum_{s=0}^da_s(q^{n-s}+q^{2n-s}+q^{3n-s}+\ldots)\\
&=&1+\frac{q^nP(q^{-1})}{1-q^n}
\end{eqnarray*}

By making the sum, we get:
$$2\int\frac{{\rm Re}(P(u^2))}{1-qu^2}\,d_nu=1+\frac{P(q)+q^nP(q^{-1})}{1-q^n}$$

The left term being by definition $1+T(q)(1-q)$, we get the result.
\end{proof}

\begin{proposition}
For $\varepsilon={\rm Re}(1-q^{2l})_n$ with $n>l$ we have:
$$T=\xi'(l,n-l:n)$$
\end{proposition}

\begin{proof}
Observe first that $\varepsilon$ is indeed a probability measure, being positive and of mass $1$. Now by applying the previous result with $P=1-q^l$, we get:
\begin{eqnarray*}
T
&=&\frac{(1-q^l)+q^n(1-q^{-l})}{(1-q)(1-q^n)}\\
&=&\frac{1-q^l-q^{n-l}+q^n}{(1-q)(1-q^n)}\\
&=&\frac{(1-q^l)(1-q^{n-l})}{(1-q)(1-q^n)}
\end{eqnarray*}

The right term being by definition $\xi'(l,n-l:n)$, we are done.
\end{proof}

\begin{theorem}
Any cyclotomic measure is an average (with real coefficients) of the basic cyclotomic measures ${\rm Re}(1-q^{2l})_n$, with $n>l$.
\end{theorem}

\begin{proof}
Let $\varepsilon$ be a cyclotomic measure, supported by the $2n$-roots. By a standard number theory argument, its $T$ series can be written in the following way:
$$T=\frac{R(q)}{(1-q)(1-q^n)}$$

Here $R\in\mathbb R[X]$ is a certain degree $n$ polynomial which is symmetric, in the sense that $R(q)=q^nR(q^{-1})$. Due to this symmetry property, we can write $R$ as an average (with real coefficients) of basic symmetric polynomials:
$$R(q)=\sum_{l=1}^nr_l(1-q^l)(1-q^{n-l})$$

Thus we have the following formula for the $T$ series:
$$T=\sum_{l=1}^nr_l\xi'(l:n-l:n)$$

Together with the formula in Proposition 4.5, this gives the result.
\end{proof}

\section{Cyclotomic expansion}

The results in the previous section suggest the following definition.

\begin{definition}
The level of a cyclotomic measure $\varepsilon$ is the smallest number $L$ such that $\varepsilon$ is an average of basic cyclotomic measures ${\rm Re}(1-q^{2l})_n$, with $l\leq L$.
\end{definition}

In this section we discuss the writing as in Theorem 4.6 (that we call cyclotomic expansion) of measures having level $L\leq 3$. This is motivated by the fact that the ADE measures have level $L\leq 3$, to be explained in detail later on.

\begin{definition}
We use the following densities:
\begin{enumerate}
\item $\alpha={\rm Re}(1-q^2)$.
\item $\beta={\rm Re}(1-q^4)$.
\item $\gamma={\rm Re}(1-q^6)$.
\end{enumerate}
\end{definition}

We know from definitions that the cyclotomic measures of level $L\leq 3$ are the linear combinations with real coefficients of measures of type $d_n,\alpha_n,\beta_n,\gamma_n$. In order to work out the uniqueness properties of such decompositions, we will use identities at the level of corresponding densities, or $T$ series.

\begin{proposition}
We have the following formulae:
\begin{enumerate}
\item For $\varepsilon=d_n$ we have $T=\xi'(n^+:n)$. 

\item For $\varepsilon=\alpha_n$ we have $T=\xi(n-1:n)$. 

\item For $\varepsilon=\beta_n$ we have $T=\xi(1^+,n-2:n)$. 

\item For $\varepsilon=\gamma_n$ we have $T=\xi'(3,n-3:n)$.
\end{enumerate}
\end{proposition}

\begin{proof}
The last formula is clear from Proposition 4.5. The middle two formulae follow as well from Proposition 4.5, by using the following identities:
\begin{eqnarray*}
\xi'(1,n-1:n)&=&\xi(n-1:n)\\
\xi'(2,n-2:n)&=&\xi(1^+,n-2:n)
\end{eqnarray*}

As for the first formula, this follows directly from Lemma 4.4.
\end{proof}

The level $0$ measures are linear combinations (with real coefficients) of measures of type $d_n$. There are several interesting questions regarding these measures:
\begin{enumerate}
\item Is there an abstract characterization of the uniformly distributed level 0 measures?

\item Is there a canonical set of level 0 measures, such that any level 0 measure is a linear combination of such measures, with positive coefficients?

\item Is there an abstract characterization of the densities of level 0 measures?
\end{enumerate}

In what follows, the level 0 measures will be simply written as linear combinations of measures of type $d_n$, with real coefficients.

We discuss now the level 1 case. Here the situation is not the same, because the system of basic measures, namely $\{d_n,\alpha_n\}$, is no longer a basis.

\begin{proposition}
The cyclotomic decomposition of a level 1 measure is unique, up to some ambiguity coming from the following exceptional equalities:
\begin{eqnarray*}
2\alpha_2&=&4d_2-2d_1\\
2\alpha_3&=&3d_3-d_1\\
2\alpha_4&=&2d_4+d_2-d_1\\
2\alpha_6&=&d_6+d_3+d_2-d_1
\end{eqnarray*}
\end{proposition}

\begin{proof}
We prove first the formulae, all of type $\alpha_n=\varepsilon_n$. In all cases both measures are supported by the $2n$-roots, and are invariant under $q\to -q$ and $q\to q^{-1}$. Thus it is enough to check that the weights at $1,w,w^2,\ldots,w^{[n/2]}$ are the same, where $w=e^{\pi i/n}$ is the first $2n$-root. The verification goes as follows:

For $n=2$ these common weights are $0,1/2$.

For $n=3$ these common weights are $0,1/4$.

For $n=4$ these common weights are $0,1/8,1/4$.

For $n=6$ these common weights are $0,1/24,1/8,1/6$.

Finally, we have to prove that this kind of decomposition can't appear for $n\neq 2,3,4,6$. We will only work out the case $n=12$, which is of a certain interest for some considerations to follow; the general case is similar.

With $w=e^{\pi i/12}$, the weights of $\alpha_{12}$ at $1,w,w^2,\ldots,w^6$ are:
$$0,\frac{2-\sqrt{3}}{48},\frac{1}{48},\frac{1}{24},\frac{3}{48},\frac{2+\sqrt{3}}{48},\frac{1}{12}$$

Now since the weights at $w,w^5$ are different, we cannot make them match just by using the measure $d_{12}$. The measures $d_m$ with $m|12$ cannot help, because they all vanish at $w,w^5$. As for the measures $d_m$ with $m\not|12$, these won't help either, because they would change the support.
\end{proof}

We discuss now the cyclotomic expansion of level 2 measures. There are many exceptional identities here, and the computations needed in order to prove them are quite long, so we will rather skip them.

\begin{proposition}
The cyclotomic decomposition of a level 2 measure is unique, up to some extra ambiguity coming from the following exceptional equalities:
\begin{eqnarray*}
2\beta_3&=&3d_3-d_1\\
2\beta_4&=&4d_4-2d_2\\
2\beta_5&=&-2\alpha_5+5d_5-d_1\\
2\beta_6&=&3d_6-d_2\\
2\beta_8&=&2d_8+d_4-d_2\\
2\beta_{10}&=&2\alpha_{10}-2\alpha_5+d_{10}+2d_5-d_2\\
2\beta_{12}&=&d_{12}+d_6+d_4-d_2
\end{eqnarray*}
\end{proposition}

\begin{proof}
The formulae follow by checking the corresponding identities at the level of densities, as in proof of Proposition 5.4, or at the level of associated $T$ series, by using the formulae in Proposition 5.3. 

As for the converse statement, this won't be used in what follows.
\end{proof}

In the level 3 case now, the combinatorics becomes quite complex, and seems to require some new ideas. We have the following result here.

\begin{proposition}
For the cyclotomic expansion of level 3 measures, we have the following exceptional equalities:
\begin{eqnarray*}
2\gamma_4&=&4d_4-d_2-d_1\\
2\gamma_5&=&-2\alpha_5+5d_5-d_1\\
2\gamma_6&=&4d_6-2d_3\\
2\gamma_8&=&2d_8+d_2-d_1\\
2\gamma_9&=&3d_9-d_3\\
2\gamma_{10}&=&-2\alpha_{10}+3d_{10}+d_5+d_2-d_1\\
2\gamma_{12}&=&2d_{12}+d_6-d_3\\
2\gamma_{18}&=&d_{18}+d_9+d_6-d_3
\end{eqnarray*}
\end{proposition}

\begin{proof}
Once again, the formulae follow either by computing the corresponding densities, or the corresponding $T$ series.
\end{proof}

Finally, we would like to record a statement about the non-allowed measures, which will appear in certain computations, later on. 

\begin{proposition}
The non-allowed measures ${\rm Re}(1-q^{2l})_n$ with $n\leq l$ and $l\leq 3$ are all null, except for $\gamma_2=2d_2-d_1$.
\end{proposition}

\begin{proof}
This is clear by computing the corresponding densities or $T$ series.
\end{proof}

\section{Binary expansion}

In principle, we are now in position of writing down the explicit formulae for the ADE measures. Indeed, we can use the following procedure:
\begin{enumerate}
\item Start with the $T$ series formulae in Theorem 2.5.

\item Decompose these as linear combinations of series in Proposition 5.3.

\item Use Propositions 5.4, 5.5, 5.6 in order to simplify the formulae.
\end{enumerate}

However, as pointed out in \cite{bdb}, some extra simplifications in the final formulae (at least for the graphs of type D) can be obtained by using measures of type $d_n$ and $d_n'=2d_{2n}-d_n$, instead of just measures of type $d_n$.

So, before getting to the ADE graphs, we have to discuss the enhanced version of the cyclotomic expansion, obtained by using measures of type $d_n'$.

We call this binary expansion. In order to study it, we will go back to the material in the previous two sections, starting with Definition 4.3, and we will work out the ``mirror versions'' of most of the statements.

\begin{definition}
We use the following notations:
\begin{enumerate}
\item $d_n'=2d_{2n}-d_n$ is the uniform measure on the odd $4n$-roots.

\item For $\lambda:\mathbb T\to\mathbb R$ we let $\lambda_n'=\lambda\,d_n'$.
\end{enumerate}
\end{definition}

Observe that the measure $\lambda_n'$ is not necessarily cyclotomic. However, for $P\in\mathbb R[X]$ the function $\lambda(q)={\rm Re}(P(q^2))$ has equal values at $q,q^{-1},-q,-q^{-1}$, so in this case $\lambda_n'$ is cyclotomic, provided that it is positive, and of mass 1.

\begin{lemma}
Let $P=1+a_1q+a_2q^2+\ldots+a_dq^d$ with $a_i\in\mathbb R$, and let $n>d$. In case $\varepsilon={\rm Re}(P(q^2))_n'$ is cyclotomic, its $T$ series is:
$$T=\frac{P(q)-q^nP(q^{-1})}{(1-q)(1+q^n)}$$
\end{lemma}

\begin{proof}
We use Lemma 4.4, along with the formula $d_n'=2d_{2n}-d_n$:
\begin{eqnarray*}
T(q)(1-q)
&=&2\,\frac{P(q)+q^{2n}P(q^{-1})}{1-q^{2n}}
-\frac{P(q)+q^{n}P(q^{-1})}{1-q^{n}}\\
&=&\frac{2(P(q)+q^{2n}P(q^{-1}))-(1+q^n)(P(q)+q^nP(q^{-1}))}{1-q^{2n}}\\
&=&\frac{(1-q^n)P(q)+(q^{2n}-q^n)P(q^{-1})}{1-q^{2n}}\\
&=&\frac{P(q)-q^nP(q^{-1})}{1+q^n}
\end{eqnarray*}

This gives the formula in the statement.
\end{proof}

\begin{proposition}
For $\varepsilon={\rm Re}(1-q^{2l})_n'$ with $n>l$ we have:
$$T=\xi'(l,n-l^+:n^+)$$
\end{proposition}

\begin{proof}
Observe first that $\varepsilon$ is indeed a probability measure, being positive and of mass $1$. Now by applying the previous result with $P=1-q^l$, we get:
\begin{eqnarray*}
T
&=&\frac{(1-q^l)-q^n(1-q^{-l})}{(1-q)(1+q^n)}\\
&=&\frac{1-q^l+q^{n-l}-q^n}{(1-q)(1+q^n)}\\
&=&\frac{(1-q^l)(1+q^{n-l})}{(1-q)(1+q^n)}
\end{eqnarray*}

The right term being by definition $\xi'(l,n-l^+:n^+)$, we are done.
\end{proof}

\begin{proposition}
We have the following formulae:
\begin{enumerate}
\item For $\varepsilon=d_n'$ we have $T=\xi'(n:n^+)$. 

\item For $\varepsilon=\alpha_n'$ we have $T=\xi(n-1^+:n^+)$. 

\item For $\varepsilon=\beta_n'$ we have $T=\xi(1^+,n-2^+:n^+)$. 

\item For $\varepsilon=\gamma_n'$ we have $T=\xi'(3,n-3^+:n^+)$. 
\end{enumerate}
\end{proposition}

\begin{proof}
The last formula is clear from Proposition 6.3. The middle two formulae follow as well from Proposition 6.3, by using the following identities:
\begin{eqnarray*}
\xi'(1,n-1^+:n^+)&=&\xi(n-1^+:n^+)\\
\xi'(2,n-2^+:n^+)&=&\xi(1^+,n-2^+:n^+)
\end{eqnarray*}

As for the first formula, this follows directly from Lemma 6.2.
\end{proof}

We have to work out now the analogues of Propositions 5.4, 5.5, 5.6. It is technically convenient to write down some lighter statements, as follows.

\begin{proposition}
For measures of type $\alpha_n'$, we have the following identities:
\begin{eqnarray*}
2\alpha_2'&=&2d_2'\\
2\alpha_3'&=&d_1'+d_3'
\end{eqnarray*}
\end{proposition}

\begin{proposition}
For measures of type $\beta_n'$, we have the following identities:
\begin{eqnarray*}
2\beta_3'&=&3d_3'-d_1'\\
2\beta_4'&=&2d_4'\\
2\beta_5'&=&2\alpha_5'+d_5'-d_1'\\
2\beta_6'&=&d_6'+d_2'
\end{eqnarray*}
\end{proposition}

\begin{proposition}
For measures of type $\gamma_n'$, we have the following identities:
\begin{eqnarray*}
2\gamma_4'&=&-2\alpha_4'+4d_4'\\
2\gamma_5'&=&-2\alpha_5'+3d_5'+d_1'\\
2\gamma_6'&=&2d_6'\\
2\gamma_9'&=&d_9'+d_3'
\end{eqnarray*}
\end{proposition}

\begin{proof}
The identities in these statements follow from those in Propositions 5.4, 5.5., 5.6, by using the formula $d_n'=2d_{2n}-d_n$.
\end{proof}

Finally, we have to discuss the case of non-allowed measures.

\begin{proposition}
The non-allowed measures ${\rm Re}(1-q^{2l})_n'$ with $n\leq l$ and $l\leq 3$ are all null, except for $\alpha_1'=2d_1'$, $\beta_2'=2d_2'$, $\gamma_1'=2d_1'$, $\gamma_2'=d_2'$, $\gamma_3'=2d_3'$.
\end{proposition}

\begin{proof}
This is clear by computing the corresponding densities or $T$ series.
\end{proof}

\section{ADE measures}

We discuss in this section the binary expansion of the circular measures of ADE graphs. These measures are all cyclotomic, of level $\leq 3$. They can be expressed in terms of the basic polynomial densities of degree $\leq 6$, namely:
\begin{eqnarray*}
\alpha&=&{\rm Re}(1-q^2)\\
\beta&=&{\rm Re}(1-q^4)\\
\gamma&=&{\rm Re}(1-q^6)
\end{eqnarray*}

We use the algorithm described in the beginning of the previous section, but with the various results in section 6 joining now those in section 5.  

We should probably recall that the problem was solved in \cite{bdb} for all graphs, except for $E_7,E_8$. As explained in the introduction, these two graphs are the main object of interest in this paper. The various computations in the previous sections are there precisely in order to be able to deal with $E_7,E_8$.

\begin{theorem}
The circular measures of ADE graphs are given by:
\begin{enumerate}
\item $A_{n-1}\to\alpha_n$.

\item $\tilde{A}_{2n}\to d_n$.

\item $D_{n+1}\to\alpha_n'$.

\item $\tilde{D}_{n+2}\to (d_n+d_1')/2$.

\item $E_6\to\alpha_{12}+(d_{12}-d_6-d_4+d_3)/2$.

\item $E_7\to\beta_9'+(d_1'-d_3')/2$.

\item $E_8\to\alpha_{15}'+\gamma_{15}'-(d_5'+d_3')/2$.

\item $\tilde{E}_{n+3}\to (d_n+d_3+d_2-d_1)/2$.
\end{enumerate}
\end{theorem}

\begin{proof}
We use the $T$ series formulae in Theorem 2.5.

The formulae for the AD graphs are already known from \cite{bdb}, and the proof goes as follows. First, the formulae for $A_n,\tilde{A}_{2n}$ and $D_n$ are clear from Proposition 5.3 and Proposition 6.4. As for the formula for $\tilde{D}_n$, this follows as well from Proposition 5.3 and Proposition 6.4, by using the following identity:
$$\xi''(n+1^+:n)=(\xi'(1:1^+)+\xi'(n^+:n))/2$$

For the graphs $\tilde{E}_6,\tilde{E}_7,\tilde{E}_8$ we use the fact that the $T$ series is given by a uniform formula. With $l=2,3,5$ corresponding to $n=6,7,8$, this series is:
\begin{eqnarray*}
\tilde{T}_n&=&\xi(3l^+:l+1,2l)\\
&=&\xi(l:l+1)+q^l\xi(:l,l+1)\\
&=&\xi(l:l+1)+(\xi'(l^+:l)-\xi'(l+1^+:l+1))/2
\end{eqnarray*}

This gives the following level 1 formula, already pointed out in \cite{bdb}:
$$\tilde{\varepsilon}_n=\alpha_{l+1}+(d_l-d_{l+1})/2$$

The level 0 formula in the statement, which is new, follows now by by using the conversion formulae in Proposition 5.4. Indeed, we have:
\begin{eqnarray*}
\alpha_3+(d_2-d_3)/2&=&(2d_3+d_2-d_1)/2\\
\alpha_4+(d_3-d_4)/2&=&(d_4+d_3+d_2-d_1)/2\\
\alpha_6+(d_5-d_6)/2&=&(d_5+d_3+d_2-d_1)/2
\end{eqnarray*}

The formulae for $E_6,E_7,E_8$ come as well by decomposing the $T$ series. The computation for $E_6$, already perfomed in \cite{bdb}, goes as follows:
\begin{eqnarray*}
\xi(8:3,6^+)
&=&\frac{1-q^8}{(1-q^3)(1+q^6)}\\
&=&\frac{1+q^3-q^8-q^{11}}{1-q^{12}}\\
&=&\frac{1-q^{11}}{1-q^{12}}+\frac{q^3-q^4-q^8+q^9}{(1-q)(1-q^{12})}\\
&=&\frac{1-q^{11}}{1-q^{12}}+\frac{1-(1+q^6)-(1+q^4+q^8)+(1+q^3+q^6+q^9)}{(1-q)(1-q^{12})}\\
&=&\frac{1-q^{11}}{1-q^{12}}+\frac{1}{1-q}\left(\frac{1}{1-q^{12}}-\frac{1}{1-q^6}-\frac{1}{1-q^4}+\frac{1}{1-q^3}\right)\\
&=&\frac{1-q^{11}}{1-q^{12}}+\frac{1}{2(1-q)}\left(\frac{1+q^{12}}{1-q^{12}}-\frac{1+q^6}{1-q^6}-\frac{1+q^4}{1-q^4}+\frac{1+q^3}{1-q^3}\right)\\
&=&\xi(11:12)+(\xi'(12^+:12)-\xi'(6^+:6)-\xi'(4^+:4)+\xi'(3^+:3))/2
\end{eqnarray*}

The computation for $E_7$ goes as follows:
\begin{eqnarray*}
\xi(12:4,9^+)
&=&\frac{1-q^{12}}{(1-q^4)(1+q^9)}\\
&=&\frac{1+q+q^7+q^8}{1+q^9}-\frac{q-q^4+q^7}{1+q^9}\\
&=&\frac{(1+q)(1+q^7)}{1+q^9}-\frac{q}{1+q^3}\\
&=&\frac{(1+q)(1+q^7)}{1+q^9}-\frac{2q-2q^2}{2(1-q)(1+q^3)}\\
&=&\frac{(1+q)(1+q^7)}{1+q^9}+\frac{1}{2(1-q)}\left(\frac{1-2q+2q^2-q^3}{1+q^3}-\frac{1-q^3}{1+q^3}\right)\\
&=&\frac{(1+q)(1+q^7)}{1+q^9}+\frac{1}{2(1-q)}\left(\frac{1-q}{1+q}-\frac{1-q^3}{1+q^3}\right)\\
&=&\xi(1^+,7^+:9^+)+(\xi'(1:1^+)-\xi'(3:3^+))/2
\end{eqnarray*}

As for the $E_8$ computation, this is a bit more complex:
\begin{eqnarray*}
&&\xi(5^+,9^+:15^+)\\
&=&\frac{(1+q^5)(1+q^9)}{1+q^{15}}\\
&=&\frac{1+q^{14}}{1+q^{15}}+\frac{2q^5-2q^6+2q^9-2q^{10}}{2(1-q)(1+q^{15})}\\
&=&\frac{1+q^{14}}{1+q^{15}}+\frac{\begin{pmatrix}(2-2q^3+2q^{12}-2q^{15})-(1-2q^5+2q^{10}-q^{15})\cr -(1-2q^3+2q^6-2q^9+2q^{12}-q^{15})\end{pmatrix}}{2(1-q)(1+q^{15})}\\
&=&\frac{1+q^{14}}{1+q^{15}}+\frac{\begin{pmatrix}2(1-q^3+q^{12}-q^{15})-(1-q^5)(1-q^5+q^{10})\cr -(1-q^3)(1-q^3+q^6-q^9+q^{12})\end{pmatrix}}{2(1-q)(1+q^{15})}\\
&=&\frac{1+q^{14}}{1+q^{15}}+\frac{1}{2(1-q)}\left(\frac{2(1-q^3+q^{12}-q^{15})}{1+q^{15}}-\frac{1-q^5}{1+q^5}-\frac{1-q^3}{1+q^3}\right)\\
&=&\frac{1+q^{14}}{1+q^{15}}+\frac{(1-q^3)(1+q^{12})}{(1-q)(1+q^{15})}-\frac{1}{2(1-q)}\left(\frac{1-q^5}{1+q^5}+\frac{1-q^3}{1+q^3}\right)\\
&=&\xi(14^+:15^+)+\xi'(3,12^+:15^+)-(\xi'(5:5^+)+\xi'(3:3^+))/2
\end{eqnarray*}

The proof of Theorem 7.1 is now complete.
\end{proof}

\section{Ternary expansion}

In this section we present the main result in this paper, obtained by further enhancing the formulae of ADE measures in Theorem 7.1.

The idea is that the $E_6,E_7,E_8$ formulae can be simplified by adding a certain measure $d_n''$ to the class of measures $\{d_n,d_n'\}$ used for binary expansions.

It is convenient to recall with this occasion the definitions of $d_n,d_n'$, and to introduce as well a certain related measure $d_n'''$. 

\begin{definition}
We use the following measures:
\begin{enumerate}
\item $d_n$ is the uniform measure on the $2n$-roots of unity.

\item $d_n'=2d_{2n}-d_n$ is the uniform measure on the odd $4n$-roots.

\item $d_n''=(3d_{3n}'-d_n')/2$ is the uniform measure on the $12n$-roots of order $6k\pm 1$.

\item $d_n'''=(3d_{3n}-d_n)/2$ is the uniform measure on the $6n$-roots of order $3k\pm 1$.
\end{enumerate}
\end{definition}

These measures already appeared in a number of formulae in previous sections. For instance we have $\alpha_3=d_1'''$, $\beta_3=d_1'''$, $\beta_3'=d_1''$, $\beta_6=d_2'''$, $\gamma_9=d_3'''$.

The point is that the support of the circular measure of $E_n$ with $n=6,7,8$ basically coincides with the support of $d_l''$ with $l=2,3,5$, in the sense that these sets are equal, up to some 2 or 4 exceptional points. Thus we can expect to have for $E_n$ a formula similar to the level 1 formula for $\tilde{E}_n$.

\begin{definition}
For $\lambda:\mathbb T\to\mathbb R$ we let $\lambda_n''=\lambda\,d_n''$, $\lambda_n'''=\lambda\,d_n'''$.
\end{definition}

Our first task is work out the explicit formulae for the basic ternary cyclotomic measures $\lambda_l''$, for densities of type $\lambda=1,\alpha,\beta,\gamma$, and for $l=2,3,5$. 

It is convenient to include as well a statement about the case $l=1$.
 
\begin{proposition}
At $l=1$ we have:
\begin{eqnarray*}
2d_1''&=&3d_3'-d_1'\\
2\alpha_1''&=&d_1''\\
2\beta_1''&=&3d_1''\\
2\gamma_1''&=&4d_1''
\end{eqnarray*}
\end{proposition}

\begin{proposition}
At $l=2$ we have:
\begin{eqnarray*}
2d_2''&=&3d_6'-d_2'\\
2\alpha_2''&=&3\alpha_6'-d_2'\\
2\beta_2''&=&d_2''\\
2\gamma_2''&=&2d_2''
\end{eqnarray*}
\end{proposition}

\begin{proof}
These formulae can be deduced by using the exceptional identities in Propositions 6.5, 6.6, 6.7, along with the extra formulae in Proposition 6.8.
\end{proof}

\begin{proposition}
At $l=3$ we have:
\begin{eqnarray*}
4d_3''&=&6d_9'-2d_3'\\
4\alpha_3''&=&6\alpha_9'-d_3'-d_1'\\
4\beta_3''&=&6\beta_9'-3d_3'+d_1'\\
4\gamma_3''&=&2d_3''
\end{eqnarray*}
\end{proposition}

\begin{proposition}
At $l=5$ we have:
\begin{eqnarray*}
4d_5''&=&6d_{15}'-2d_5'\\
4\alpha_5''&=&6\alpha_{15}'-2\alpha_5'\\
4\beta_5''&=&6\beta_{15}'-2\alpha_5'-d_5'+d_1'\\
4\gamma_5''&=&6\gamma_{15}'+2\alpha_5'-3d_5'-d_1'
\end{eqnarray*}
\end{proposition}

\begin{proof}
These formulae can be deduced as well from those in section 6.
\end{proof}

We are now in position of stating the main result. The idea is to modify Theorem 7.1, by keeping fixed the formulae for AD graphs, and by modifying those for E graphs, by using the notion of ternary expansion.

\begin{theorem}
With $l=2,3,5$ corresponding to $n=6,7,8$, the circular measures of the ADE graphs are given by:
\begin{enumerate}
\item $A_{n-1}\to\alpha_n$.

\item $\tilde{A}_{2n}\to d_n$.

\item $D_{n+1}\to\alpha_n'$.

\item $\tilde{D}_{n+2}\to (d_n+d_1')/2$.

\item $E_6\to (d_2''+2\alpha_2''+3d_1''')/6$.

\item $E_7\to (2\beta_3''+d_1')/3$.

\item $E_8\to (2\alpha_5''+2\gamma_5''-d_1'')/3$.

\item $\tilde{E}_n\to \alpha_{l+1}+(d_l-d_{l+1})/3$.
\end{enumerate}
\end{theorem}

\begin{proof}
The formulae for AD graphs are those in Theorem 7.1. For $E_6$ we have:
\begin{eqnarray*}
12\varepsilon_6
&=&12\alpha_{12}+6d_{12}-6d_6-6d_4+6d_3\\
&=&(6\alpha_6+6\alpha_6')+(3d_6+3d_6')-6d_6-6d_4+6d_3\\
&=&6\alpha_6+6\alpha_6'-3d_6+3d_6'-6d_4+6d_3\\
&=&(3d_6+3d_3+3d_2-3d_1)+6\alpha_6'-3d_6+3d_6'-(3d_2+3d_2')+6d_3\\
&=&9d_3-3d_1+6\alpha_6'+3d_6'-3d_2'\\
&=&(3d_6'-d_2')+(6\alpha_6'-2d_2')+(9d_3-3d_1)\\
&=&2d_2''+4\alpha_2''+6d_1'''
\end{eqnarray*}

For $E_7,E_8$ the ternary conversion is much simpler, because the formulae in Theorem 7.1 are already in binary form. For $E_7$ we have:
\begin{eqnarray*}
6\varepsilon_7
&=&6\beta_9'+3d_1'-3d_3'\\
&=&(6\beta_9'-3d_3'+d_1')+2d_1'\\
&=&4\beta_3''+2d_1'
\end{eqnarray*}

As for $E_8$, we have:
\begin{eqnarray*}
6\varepsilon_8
&=&6\alpha_{15}'+6\gamma_{15}'-3d_5'-3d_3'\\
&=&(6\alpha_{15}'-2\alpha_5')+(6\gamma_{15}'+2\alpha_5'-3d_5'-d_1')-(3d_3'-d_1')\\
&=&4\alpha_5''+4\gamma_5''-2d_1''
\end{eqnarray*}

Finally, the formula for $\tilde{E}_n$ is the level 1 one in proof of Theorem 7.1.
\end{proof}

\section{Concluding remarks}

The probabilistic study of quantum algebraic invariants of the ADE graphs, motivated by the results in \cite{jo2}, started in \cite{bdb}, and continued in this paper, leads naturally to the notions of cyclotomic measure, level, and expansion. 

The cyclotomic measures are those supported by the roots of unity, which are invariant under the reflections with respect to the real and imaginary axes. The notions of cyclotomic level and expansion come from the various ways of writing such a measure, by using the combinatorial properties of the roots of unity.

The circular spectral measures of the finite ADE graphs are all cyclotomic, of level $\leq 3$. The cyclotomic expansions of these measures are as follows:
\begin{enumerate}
\item For the $A$ graphs we have a uniform formula. 

\item For the $D$ graphs we have a uniform binary formula. 

\item For the $\tilde{E}$ graphs we have a uniform formula in terms of $n=6,7,8$, and a uniform formula in terms of $l=2,3,5$.

\item For the $E$ graphs, which have a considerably more complicated combinatorics, we have a ternary formula, in terms of $l=2,3,5$.
\end{enumerate}

In addition, the circular spectral measures of the various infinite ADE graphs are simply obtained by taking limits in the finite case formulae.

As explained in the introduction, some other interesting measures supported by the unit circle appear in the recent litterature. Some of them, as the measure studied in \cite{bcz} in connection with the quantum analogue of the Weingarten problem \cite{wei}, are of infinite level, but their cyclotomic writing doesn't require the use of binary or ternary expansions, as needed in the ADE case.

Summarizing, the general combinatorial problem which seems to emerge from these considerations, and that we would like to raise here, is that of understanding the notion of ternary expansion, for an arbitrary cyclotomic measure. 

One can hope for instance that such a general result might help in further understanding the $D_{odd}$ and $E_7$ subfactor obstructions.

The other problem, already raised in \cite{bdb}, is to find the correct analogue of the circular measure, in the index $>4$ case. This seems to require a good knowledge of the arithmetic properties of the eigenvalues of the principal graph. In other words, we are in front of a well-known, difficult problem, namely the extension and possible unification of the obstructions in \cite{eno} and \cite{jo2}.

\end{document}